\pdfoutput=1

\documentclass[12pt]{amsart}
\usepackage{Alegreya}

\newtheorem{theorem}{Theorem}[subsection]
\newtheorem{lemma}[theorem]{Lemma}

\theoremstyle{definition}   
\newtheorem{definition}[theorem]{Definition}

\theoremstyle{remark}
\newtheorem{remark}[theorem]{Remark}

\numberwithin{equation}{section}
\usepackage{enumerate}
\usepackage{tikz-cd}
\usetikzlibrary{bending}
\usepackage{graphicx,adjustbox}
\usepackage{hyperref}
\usepackage{amsmath,amssymb}
\usepackage[normalem]{ulem}
\usepackage{upgreek}
\hypersetup{colorlinks,linkcolor={blue},citecolor={red},urlcolor={red}}
\usepackage[a4paper,
            bindingoffset=0.2in,
            left=1in,
            right=1in,
            top=1in,
            bottom=1in,
            footskip=.25in]{geometry}

\usepackage[
backend=biber,
bibstyle=numeric-comp,
sorting=nyt,doi=false,url=false,
]{biblatex}
\DeclareNameAlias{author}{last-first}
\DeclareFieldFormat[article, inbook]{title}{#1} 
\renewbibmacro{in:}{}
\DeclareFieldFormat{journaltitle}{#1\isdot}

\usepackage{amscd}
\usepackage{graphicx}
\usepackage[all]{xy}
\title[Co-Higgs bundles of Schwarzenberger type and the determinant morphism]{Co-Higgs bundles of Schwarzenberger type and the determinant morphism}
\author{KUNTAL BANERJEE}
\address{Department of Pure Mathematics University Waterloo Canada}
\email{kub129@usask.ca}

\begin{document}
\begin{abstract}
We identify images of the determinant morphism of trace-free co-Higgs bundles modelled on rank $2$ Schwarzenberger bundles.
\end{abstract}
\maketitle 

\let\thefootnote\relax\footnotetext{2020\textit{ Mathematics Subject Classification.} 14J60}\let\thefootnote\relax\footnotetext{}\let\thefootnote\relax\footnotetext{}

\section{Introduction}
\subsection{Overview}

Let $X$ be a complex algebraic curve of genus $g_X\geq 2$. The Hitchin morphism $\mathcal{H}$ of Higgs bundles on $X$ is proper and surjective onto the Hitchin base $\mathcal{B}$ while a generic fiber is an abelian variety and a completely integrable system. When $X$ is a variety of a higher dimension $\mathcal{H}$ is not surjective though it is proper. The conjectured image of the Hithin morphism is a nonlinear closed subscheme $\mathcal{B}^{\heartsuit}\subset\mathcal{B}$ such that $\mathcal{H}$ factors through the inclusion of $\mathcal{B}^{\heartsuit}\subset\mathcal{B}$ and there exists an open dense subscheme $\mathcal{B}^{\diamondsuit}\subset \mathcal{B}^{\heartsuit}$ that is contained inside the image of $\mathcal{H}$ (\cite{Bao}). In particular, the spectral covering map $\pi_s: X_s\to X$ for $s\in\mathcal{B}^{\heartsuit}$ is finite and the spectral correspondence sets a bijective correspondence between the twisted vector bundles with spectral cover $X_s$ and Cohen-Macaulay sheaves of generic rank $1$ on $X_s$. The conjecture is verified for $\dim X = 2$ on vector bundles (\cite{hao}) and for principal Higgs bundles of certain classical Lie groups (\cite{huynh}). A conjectured image of $\mathcal{H}$ is proposed naturally for $V$-twisted Higgs bundles of rank $2$ on algebraic curves (\cite{Gallego}). Thus far no description more explicit than these conjectured images, where the twist is not a line bundle, is available in the literature. To overcome this difficulty we aim at exploring the Hitchin morphism restricted on twisted Higgs bundles whose underlying vector bundles are well-selected and fixed. Throughout this paper we will call such twisted Higgs bundles \textit{modelled} on their underlying bundles. For example, we can consider co-Higgs bundles or $T_{\mathbb{P}^2}$-twisted Higgs bundles modelled on rank $2$ Schwarzenberger bundles. Note that in the rank $2$ case $\mathcal{H}$ consists of a linear component --- the \textit{trace} component, and a nonlinear component --- the \textit{determinant} component. Moreover, $\mathcal{H}$ is represented by the determinant component if the trace component is shifted to $0$. So we focus on trace-free twisted Higgs bundles and study their determinants to identify the image of $\mathcal{H}$. The main results in this paper (Theorem \ref{det1}, Theorem \ref{det2}, Theorem \ref{det3}, and Theorem \ref{det4}) compute images of the determinant morphism of Gieseker semistable traceless co-Higgs bundles modelled on Schwarzenberger bundles $V_k^\rho$ for $k\neq 3$.\\ 

Co-Higgs bundles are important geometric objects on multiprojective spaces and stable co-Higgs bundles vanish on algebraic surfaces of general type and so on \cite{Ste3}. For a fixed nonsingular conic $\rho$ inside $\mathbb{P}^2$, Schwarzenberger bundles are a sequence of holomorphic vector bundles $\{V_k^\rho\}_{k=1}^{\infty}$ of rank $2$ on $\mathbb{CP}^2$ coming from a $2:1$ covering map $f^\rho: \mathbb{P}^1\times\mathbb{P}^1\to\mathbb{P}^2$ branched over $\rho$ (\cite{sch}, \cite{Fried}). Let $k\neq 3$. If $\phi$ is a nonzero stable traceless co-Higgs bundle modelled on $V_k^\rho$ then $\phi = \phi_0\otimes C$ for some  
$\phi_0\in H^0(\mbox{End}_0(E)\otimes\mathcal{O}(1))$ and $C\in H^0(T_{\mathbb{P}^2}(-1))$. This characterization of stable co-Higgs bundles along with the standard Euler sequence for $T_{\mathbb{P}^2}$ is utilized to construct moduli spaces and to find dimensions of spaces of their infinitesimal deformations (\cite{Ste3}). In particular, $\det(\phi) = \det(\phi_0)\otimes\mbox{Sym}^2(C)$ where $\mbox{Sym}^2(C)\in H^0(\mbox{Sym}^2(T_{\mathbb{P}^2})\otimes\mathcal{O}(-2))$ and $\det(\phi_0)\in H^0(\mathcal{O}(2))$. This description enables us to write down complex schemes that parametrize the images of the determinant morphism. On $V_3^\rho$ there are stable co-Higgs bundles which are not of such type. So we exclude the case $k=3$ completely from our analysis.

\subsection{Acknowledgments}
The content of the subsection 2.2 appeared within my doctoral thesis (\cite{Banerjee}). I am grateful to my PhD supervisor, Steven Rayan, for insightful guidance and continuous support, both during the thesis work and during an additional $6$-month postdoctoral fellowship that I held under his mentorship. I am fortunate to have been supported financially by the Department of Mathematics and Statistics, the College of Arts and Science and the College of Graduate and Postdoctoral Studies at the University of Saskatchewan though a Graduate Teaching Fellowship during my doctoral studies. I was also generously supported by S. Rayan's NSERC Discovery Grant during both my doctoral and postdoctoral work. I am thankful for many fruitful mathematical conversations with Mahmud Azam, Eric Boulter, Robert Cornea, Matthew Koban, Dat Minh Ha, and Evan Sundbo. Finally, I acknowledge Ruxandra Moraru for hosting me at the University of Waterloo through an additional postdoctoral fellowship in Winter 2025, during which time I made significant progress on this manuscript. Finally, I thank the anonymous reviewer for sharing meaningful suggestions towards improvement of the manuscript. 
\section{The determinant morphism on co-Higgs bundles of Schwarzenberger type}

\subsection{Preliminaries}

Let $X$ be an $n$-dimensional smooth connected complex projective algebraic variety with structure sheaf $\mathcal{O}_X$ and $H$ be a fixed ample line bundle on $X$. Let $E$ be a coherent sheaf of pure dimension $d = \dim (E)$ and rank $r$ on $X$. There are unique integers $\alpha_0(E),\dots, \alpha_d(E)$ with $\alpha_d(E) > 0$ such that the Hilbert polynomial of $E$ with respect to $H$, namely, 
\begin{equation}
P(E, m) := \chi\left(E\otimes H^m\right), 
\end{equation}
satisfies 
\begin{equation}
P(E, m) = \sum_{i = 0}^d \frac{\alpha_i(E)}{i!}m^i, ~~\forall m\geq 0.
\end{equation}
A coherent sheaf $E$ is said to be \textit{Gieseker (semi)stable} if for each nonzero proper coherent subsheaf $F\subset E$,
\begin{equation}
p(F) = \frac{P(F, m)}{r(F)}(\leq) < \frac{P(E, m)}{r(E)} = p(E)~~\forall m>> 0.
\end{equation}
If $E$ is torsion free then the \textit{degree} of $E$ is defined as the integer $\deg(E) := \alpha_{n-1}(E) - r(E)\cdot\alpha_{n-1}(E).$ It follows that $\deg(E) = c_1(E)\cdot c_1(H)^{n-1}\in H^{2n}(X, \mathbb{Z})\cong\mathbb{Z}$. Furthermore, $E$ is said to be \textit{slope (semi)stable} if 
\begin{equation}
\frac{\deg(F)}{r(F)}(\leq) <\frac{\deg(E)}{r(E)}.
\end{equation}
Note that $E$ is slope-stable $\implies$ $E$ is Gieseker stable $\implies$ $E$ is Gieseker semistable $\implies$ $E$ is slope semistable.\\

Let $E_1, E_2$ be Gieseker semistable. If $p(E_1) > p(E_2)$ then $H^0(\mbox{Hom}(E_1, E_2)) = 0$. If $E_1$ is Gieseker stable and $p(E_1) = p(E_2)$ then $H^0(\mbox{Hom}(E_1, E_2)) = 0$ or $E_1\cong E_2$. If $E$ is Gieseker stable then $E$ is simple i.e. $H^0(\mbox{End}(E))\cong\mathbb{C}$.\\

A pure coherent sheaf $E$ of dimension $d$ admits a filtration of coherent subsheaves 
\begin{equation}
0 = E_0\subsetneq E_1\subsetneq\dots\subsetneq E_l = E 
\end{equation}
such that $E_i/E_{i-1}$ is stable, pure of dimension $d$ and $p(E_i/E_{i-1}) = p(E)$. This is said to be a Jordan-H\"older filtration of $E$. If $E$ is Gieseker semistable then $E$ admits a Jordan-H\"older filtration. The \textit{grading} of $E$ is the direct sum $\mbox{gr}(E) = \oplus_i ~(E_i/E_{i-1})$ unique up to an isomorphism. Two Gieseker semistable sheaves $E$ and $F$ are \textit{strongly equivalent} if $\mbox{gr}(E)\cong\mbox{gr}(F)$.
In particular, if $E$ and $F$ are Gieseker stable then $E$ and $F$ are strongly equivalent if and only if isomorphic as sheaves.\\

Gieseker (semi)stable coherent sheaves form a moduli scheme (. \cite{Simp1}, \cite{Simp2}). Appealing to a projective embedding of $X$, it is shown at first that the Gieseker semistable sheaves with a fixed Hilbert polynomial $P$ are bounded (\cite{Simp1} page 56). Therefore, the functor $\mbox{\textit{Sch}}/\mathbb{C}\to \mbox{\textit{Sets}}$ sending a $\mathbb{C}$-scheme $S$ to the set of strong equivalence classes of families of Gieseker semistable sheaves with Hilbert polynomial $P$ parametrized by $S$ is universally corepresented by a projective scheme $\mathcal{M}_{X, H}^{ss}(P)$ (\cite{Simp1} page 65). There is an open quasi-projective subscheme $\mathcal{M}_{X, H}^{s}(P)\subset\mathcal{M}_{X, H}^{ss}(P)$ that parametrizes isomorphism classes of Gieseker stable sheaves. Unless $X$ is a curve the scheme $\mathcal{M}_{X, H}^{s}(P)$ is not necessarily smooth. At a Gieseker stable point $E$ if the space of obstructions vanishes i.e.  $H^2(X, \mbox{End}(E)) = 0$ then $E$ is a smooth point and the Zariski tangent space is $H^1(X, \mbox{End}(E))$.\\

Twisted Higgs sheaves are coherent sheaves with a new module structure (\cite{Nitin}, \cite{Simp1}, \cite{Simp2}, \cite{Gallego}). Let $V$ be a vector bundle. By a \textit{$V$-twisted Higgs sheaf $(E, \phi)$} we mean a pure coherent sheaf $E$ and a morphism of sheaves $\phi:E\to E\otimes V$, in other words, $\phi\in H^0(\mbox{Hom}(E, E\otimes V))$ which satisfies $\phi\wedge\phi = 0$. If $E$ is locally free then $(E, \phi)$ is said to be a \textit{$V$-twisted Higgs bundle}. The morphism $\phi$ is said to be a \textit{$V$-twisted Higgs field}. The last condition on $\phi$ is said to be the \textit{integrability condition} on $\phi$ (Lemma 2.12, Lemma 2.13 \cite{Simp1}). If $E$ is locally free of rank $r$ one can write 
\begin{equation}
\phi = \sum\limits_{i=1}^r\phi_i\otimes u_i,~~ \phi_i\in\mathcal{O}(\mbox{End}E)(\mathcal{U}),~~
\end{equation}
where $\mathcal{U}$ is a common trivializing neighborhood of $E$ and $V$ and $\{u_i\}$ is a local basis of $V$. Then integrability can be phrased as 
\begin{equation}\label{inte}
\phi\wedge\phi = \sum\limits_{i<j} [\phi_i, \phi_j]\otimes (u_i\wedge u_j) = 0.
\end{equation}

The equation \ref{inte} reads: $\phi$ is integrable if and only if the $\phi_i$'s commute with each other. If $V$ or $E$ is a line bundle then \ref{inte} is immediate. 

A homomorphism between $V$-twisted Higgs sheaves $(E_1, \phi_1)$ and $(E_2, \phi_2)$ is a commutative diagram as follows
\begin{equation}
\begin{tikzcd}
E_1 \arrow[r, "\phi_1"] \arrow[d, "\psi"]
& E_1\otimes V \arrow[d, "\psi \otimes 1 = \psi' "] \\
E_2 \arrow[r, "\phi_2 "]
& E_2\otimes V
\end{tikzcd}.
\end{equation}
In particular, $(E_1, \phi_1)$ and $(E_2, \phi_2)$ are said to be \textit{isomorphic} if $\psi$ is an isomorphism. In that case we write $(E_1, \phi_1)\cong (E_2, \phi_2)$.\\

A coherent subsheaf $F\subset E$ is said to be \textit{invariant} under $\phi$ if $\phi(F)\subseteq F\otimes V$. A coherent twisted sheaf $E$ is said to be Gieseker (semi)stable if for each invariant coherent subsheaf $F\subsetneq E$, it holds that $p(F) (\leq) < p(E)$. A Gieseker semistable $V$-twisted sheaf $(E, \phi)$ admits a Jordan-H\"older filtration of $V$-twisted subsheaves. Let $(E, \phi)$ be a Gieseker semistable pure coherent sheaf of dimension $d$. Then there exists a strict filtration of coherent $V$-twisted subsheaves 
\begin{equation}
0 = E_0\subsetneq E_1\subsetneq\dots\subsetneq E_l = E 
\end{equation}
such that $(E_i/E_{i-1}, \phi: E_i/E_{i-1}\to E_i/E_{i-1}\otimes V)$ is stable, pure of dimension $d$ and $p(E_i/E_{i-1}) = p(E)$. Two semistable twisted sheaves are said to be \textit{strongly equivalent} if their grading sheaves are isomorphic to each other. 

An invariant theoretic quotient of twisted Higgs sheaves is given with an extended construction of a quotient of pure semistable sheaves. There is a complex quasi-projective scheme $\mathcal{M}_{X, H}^{ss} (P, V)$ (\cite{Simp1}) that consists of the strongly equivalent semistable $\mbox{Sym}(V^*)$-modules and there exists an open subscheme $\mathcal{M}_{X, H}^{s}(P, V)\subset\mathcal{M}_{X, H}^{ss}(P, V)$ which parametrizes stable $\mbox{Sym}(V^*)$-modules.\\

There is another useful construction for twisted Higgs sheaves in terms of coherent sheaves on a projective variety. First, let $\pi:V\to X$ be the bundle projection map and $(E, \phi)$ be a $V$-twisted torsion free Higgs sheaf. Then $(E, \phi) \cong \pi_*(\mathcal{E}, \eta)$ where $\mathcal{E}$ is a coherent sheaf on $\mbox{Tot}(V)$ of dimension $n$ and $\eta$ is the tautological section of $\pi^*V$. Let $Z = \mathbb{P}(V^*\oplus\mathcal{O}) = \overline{\mbox{Tot}(V)}$. Denote the extension map of $\pi$ also with $\pi:Z\to X$ and with $D\subset Z$ the hyperplane divisor. Then $\exists k\in\mathbb{N}$ such that $H' = \pi^*H^k\otimes\mathcal{O}_Z(D)$ is ample on $Z$. Then $H'|_{Z\backslash D} = \pi^*H^k$. The sheaf $\mathcal{E}$ is a coherent sheaf on $Z$ supported outside $D$ and $\chi(\mathcal{E}\otimes(\pi^*H^k)^m) = \chi(\pi_*\mathcal{E}\otimes (H^k)^m)$ for all integers $m$. In particular, $(E, \phi)$ is Gieseker (semi)stable with respect to $H^k$ if and only if $\mathcal{E}$ is Gieseker (semi)stable with respect to $\pi^*H^k$. Here $\mathcal{M}_{X, H^k}^{ss}(P, V)$ is an open subscheme of $\mathcal{M}_{Z, H'}^{ss}(P)$ consisting of Gieseker semistable sheaves with dimension $\dim(E)$ supported outside $D$.\\ 

An integrable $V$-twisted Higgs field $\phi$ defines a chain complex of sheaves (\cite{Simp2}, \cite{Ste3})
\begin{equation}
C^\bullet: \mbox{End}E\xrightarrow[]{\wedge\phi}\mbox{End}E\otimes V\xrightarrow[]{\wedge\phi}\mbox{End}E\otimes\wedge^2 V.
\end{equation}
which fits into a specific exact sequence 
\begin{equation}
0\to\mathcal{E}^{1,0}\to T_{(E,\phi)} =\mathbb{H}^1(C^\bullet)\to\mathcal{E}^{0,1}\to\mathcal{E}^{2,0}\to \mathbb{H}^2(C^\bullet) 
\end{equation}

where 
\begin{equation}
\mathcal{E}^{p,q} = \frac{\ker\left(H^q(\mbox{End}E\otimes\wedge^pT_X)\xrightarrow[]{\wedge\phi} H^q(\mbox{End}E\otimes\wedge^{p+1}T_X\right)}{\mbox{Im}\left(H^q(\mbox{End}E\otimes\wedge^{p-1}T_X)\xrightarrow[]{\wedge\phi} H^q(\mbox{End}E\otimes\wedge^{p}T_X\right)}.\end{equation}

Let $P$ be a Hilbert polynomial with degree $d = n$ and $(E, \phi)$ is a torsion free $V$-twisted Higgs sheaf with $P(E) = P$. There exists an open dense subset $\mathcal{U}\subseteq X$ such that $E_{|\mathcal{U}}$ is a vector bundle and $\phi_{|\mathcal{U}}:E_{|\mathcal{U}}\to E_{|\mathcal{U}}\otimes V_{|\mathcal{U}}$ is a homomorphism of vector bundles and the characteristic coefficients are global sections of $\mbox{Sym}^iV_{|\mathcal{U}}$ for $i = 1,\dots, r$. By Hartogs' theorem the characteristic coefficients extend to global sections of $\mbox{Sym}^iV$ for each $i$. These sections are said to be \textit{the characteristic coefficients} of $(E, \phi)$. 
\begin{definition}\label{hitchmo}
The Hitchin morphism $\mathcal{H}:\mathcal{M}_{X, H}^{ss}(P, V)\to \mathcal{B}(r)=\oplus_{i=1}^r H^0(X, \mbox{Sym}^iV)$ computes the characteristic coefficients. The determinant morphism $\mbox{det}$ is the composition of the natural projection morphism $\mathcal{B}(r)\to H^0(X, \mbox{Sym}^rV)$ with $\mathcal{H}$. In particular, if $V = L$ where $L$ is a line bundle and $(E, \phi)\in\mathcal{M}_{X, H}^{ss}(P, L)$ then $\det(\phi)\in H^0(L^r)$.
\end{definition}

In \ref{hitchmo} it suffices to assume that $E$ is a vector bundle on $X$. If $(E, \phi:E\to E\otimes V)$ is a twisted torsion free sheaf then choose $(E^{**}, \phi^{**}: E^{**}\to E^{**}\otimes V)$, so that $P(E) = P(E^{**})$ and $\mathcal{H}(E, \phi) = \mathcal{H}(E^{**}, \phi^{**})$ and there is a commutative diagram as follows
\begin{equation}
\begin{tikzcd}
E \arrow[r, "\phi"] \arrow[d, "\iota"]
& E\otimes V \arrow[d, "\iota \otimes 1"] \\
E^{**} \arrow[r, "\phi^{**}"]
& E^{**}\otimes V
\end{tikzcd}.
\end{equation}

\begin{theorem}
(\cite{gallego2023higgs}, \cite{Simp2}) Hitchin morphism $\mathcal{H}: \mathcal{M}_{X, H^k}^{ss}(P, V)\to \mathcal{B}(r)=\oplus_{i=1}^r H^0(X, \emph{Sym}^iV)$ is proper. 
\end{theorem}

\subsection{Schwarzenberger bundles and determinant morphism}\label{cohiggs}
\begin{definition}
For each nonzero irreducible polynomial $\rho\in H^0(\mathbb{P}^2, \mathcal{O}(2))$, there exists a holomorphic cover $f^\rho:\mathbb{P}^1\times\mathbb{P}^1\to\mathbb{P}^2$ branched over the nonsingular conic $\rho = 0$. The bundle $V_k^\rho = f^\rho_*\mathcal{O}(0, k)$ where $\mathcal{O}(0, k) = \mathcal{O}\boxtimes\mathcal{O}(k)$ is said to be the $k$-th Schwarzenberger bundle. 
\end{definition}
For $k = 0, 1, 2$ the bundle $V_k^\rho$ is rigid, that is, $H^1(\mbox{End}(V_k^\rho)) = 0$. For $k\geq 3$, it follows that  $H^1(\mbox{End}_0(V_k^\rho))\cong \mathbb{C}^{k^2 - 4}$. For every $\rho$, it follows that 
\begin{equation}
V_k^\rho =\begin{cases}
   \mathcal{O}\oplus\mathcal{O}(-1);~ k = 0\\
   \mathcal{O}\oplus\mathcal{O};~ k=1\\
   T_{\mathbb{P}^2};~ k=2.
\end{cases}\end{equation} On the other hand, if $V_3^\rho\cong V_3^{\rho'}$ then the zero schemes $\rho = 0$ and $\rho' = 0$ coincide. Furthermore, $c_1(V_k^\rho) = (k-1)H$ and $c_2(V_k^\rho) = \frac{k(k-1)}{2}H^2$. Co-Higgs bundles on $V_k^\rho$ are stable under infinitesimal deformations in a sense that the underlying bundle of a deformed co-Higgs bundle is another Schwarzenberger bundle. The map $f^\rho$ also offers a spectral viewpoint on $\mathbb{P}^2$ which is briefly highlighted in \cite{Ste3}.\\

Let $\mathcal{M}_{\mathbb{P}^2}(V_k^\rho, T_{\mathbb{P}^2})$ denote the space of trace-free Gieseker semistable co-Higgs bundles modelled on $V_k^\rho$ and $\det|_{\mathcal{M}_{\mathbb{P}^2}(V_k^\rho, T_{\mathbb{P}^2})}: {\mathcal{M}_{\mathbb{P}^2}(V_k^\rho, T_{\mathbb{P}^2})}\to H^0(\mbox{Sym}^2(T_{\mathbb{P}^2}))$ be the determinant morphism. We will identify $\mbox{Im}(\det|_{\mathcal{M}_{\mathbb{P}^2}(V_k^\rho, T_{\mathbb{P}^2})})$ with known spaces. The chosen polarization for our computations is $H = \mathcal{O}_{\mathbb{P}^2}(1)$. First we prove some elementary lemmas which will be necessary in our analysis.

\begin{lemma}
Let $C, C' \in H^0(T_{\mathbb{P}^2}(-1))$ be such that $C\otimes C = C'\otimes C'\in H^0(T_{\mathbb{P}^2}\otimes T_{\mathbb{P}^2}\otimes\mathcal{O}(-2))$. Then $C=\pm C'$.
\end{lemma}
\begin{proof}
The strategy of the proof makes sense for any section $C\in H^0(V)\backslash\{0\}$ where $V$ is a bundle of rank $2$. Let $U_0 = \mathbb{C}^2$ be a standard coordinate chart on $\mathbb{P}^2$. On $U_0$ write $C$ and $C'$ as $C = s_1e_1 + s_2e_2$ and $C' = s_1'e_1 + s_2'e_2$. Here $s_1, s_2, s_1', s_2'$ are holomophic functions on $U_0$ and $\{e_1, e_2\}$ is a local basis of sections of $T_{\mathbb{P}^2}(-1)$. Then $C\otimes C = C'\otimes C'$ implies that \begin{equation}
s_1^2 = s_1'^2;~ s_1s_2 = s_1's_2';~ s_2^2 = s_2'^2.
\end{equation}
These three equations together imply that either $C = C'$ or $C = - C'$ throughout on $U_0$. Either of these equalities extends to $C = C'$ or $C = - C'$ throughout $\mathbb{P}^2$.
\end{proof}

\begin{lemma}
Every section $C\in H^0(T_{\mathbb{P}^2}(-1))$ defines a global holomorphic section $\emph{Sym}^2(C)\in H^0(\emph{Sym}^2(T_{\mathbb{P}^2})\otimes\mathcal{O}(-2))$. Therefore, $\emph{Sym}^2(C) = \emph{Sym}^2(C')$ if and only if $C\otimes C = C'\otimes C'$.
\end{lemma}

\begin{proof}
The strategy of the proof makes sense for any section $C\in H^0(V)$ where $V$ is a bundle of rank $2$. The section $\mbox{Sym}^2(C)\in H^0(\mbox{Sym}^2(T_{\mathbb{P}^2})\otimes\mathcal{O}(-2))$ is defined locally as $s_1^2e_1^2 + 2s_1s_2e_1e_2 + s_2^2e_2^2$. Let $g_{\alpha\beta}$ be a transition function of $T_{\mathbb{P}^2}(-1)$ given by \begin{equation}
   g_{\alpha\beta} = \begin{bmatrix}
        g_{\alpha\beta}^{11} & g_{\alpha\beta}^{12}\\
        g_{\alpha\beta}^{21} & g_{\alpha\beta}^{22}
    \end{bmatrix}. 
\end{equation} Then \begin{equation}
    \mbox{Sym}^2(g_{\alpha\beta}) = \begin{bmatrix}
        {g_{\alpha\beta}^{11}}^2 & g_{\alpha\beta}^{11}g_{\alpha\beta}^{12} & {g_{\alpha\beta}^{12}}^2\\
        2g_{\alpha\beta}^{21}g_{\alpha\beta}^{11} & g_{\alpha\beta}^{22}g_{\alpha\beta}^{11} + g_{\alpha\beta}^{21}g_{\alpha\beta}^{12} & 2g_{\alpha\beta}^{22}g_{\alpha\beta}^{12}\\
        {g_{\alpha\beta}^{21}}^2 & g_{\alpha\beta}^{21}g_{\alpha\beta}^{22} & {g_{\alpha\beta}^{22}}^2
    \end{bmatrix}
\end{equation} with respect to the ordered basis $\{e_1^2, e_1e_2, e_2^2\}$. The local representatives of $\mbox{Sym}^2(C)$ on trivializing neighborhoods agree with the effect of the transition data $\mbox{Sym}^2(g_{\alpha\beta})$ of $\mbox{Sym}^2(T_{\mathbb{P}^2})\otimes\mathcal{O}(-2)$. The second part follows immediately from the definition of $\mbox{Sym}^2(C)$.
\end{proof}

\begin{theorem}\label{det1}
There is a bijection $\emph{Im}\left(\det|_{\mathcal{M}_{\mathbb{P}^2}(V_0^\rho, T_{\mathbb{P}^2})}\right)\to\frac{H^0(\mathcal{O}(2))^\times\times H^0(T_{\mathbb{P}^2}(-1))^\times}{(q, C)\sim (\alpha^2\cdot q, \alpha^{-1}\cdot C)}\sqcup\{0\}$ given by $q\otimes\mbox{Sym}^2(C)\mapsto [(q, C)]$.
\end{theorem}
\begin{proof}
Since $V_0^\rho = \mathcal{O}\oplus\mathcal{O}(-1)$ has coprime rank and degree a traceless co-Higgs field $\phi:\mathcal{O}\oplus\mathcal{O}(-1)\to \left(\mathcal{O}\oplus\mathcal{O}(-1)\right)\otimes T_{\mathbb{P}^2}$ is Gieseker semistable if and only if Gieseker stable. It follows from integrability of $\phi$ and Hartogs' theorem that (\cite{Rayan}) 
\begin{equation}
\phi = \begin{bmatrix}
    \lambda & \mu\\
    1 & -\lambda
\end{bmatrix}\otimes C,~~ \lambda\in H^0(\mathcal{O}(1)), \mu\in H^0(\mathcal{O}(2)), C\in H^0(T_{\mathbb{P}^2}(-1)).
\end{equation}
Thus $\det(\phi) = -(\lambda^2 + \mu)\otimes\mbox{Sym}^2(C)$. So,
\begin{align*}
&\mbox{Im}\left(\det|_{\mathcal{M}_{\mathbb{P}^2}(V_0^\rho, T_{\mathbb{P}^2})}\right) \\
& = \left\{(\lambda^2 + \mu)\otimes\mbox{Sym}^2(C):\lambda\in H^0(\mathcal{O}(1)),~\mu\in H^0(\mathcal{O}(2)),~ C\in H^0(T_{\mathbb{P}^2}(-1))\right\}.
\end{align*}
Now
\begin{equation}
\phi_0 = \begin{bmatrix}
    \lambda & \mu\\
    1 & -\lambda
\end{bmatrix}\mapsto -(\lambda^2 + \mu)
\end{equation} is a surjection onto $H^0(\mathcal{O}(2))$. A global holomorphic section of $\mathcal{O}(1)$ is presented by a polynomial $\lambda = az + bw + c$ of total degree $\leq 1$ and a global holomorphic section of $\mathcal{O}(2)$ is presented by a complex polynomial $s = az^2 + bzw + cw^2 + dz + ew + f$ of total degree $\leq 2$. It follows that $s = \lambda^2 + \mu$. This is an elementary verification via adjusting the coefficients of $s$.\\

(i) Let $a\neq 0$. Then 
\begin{align*}
az^2 + bzw + cw^2 + dz + ew + f = \left(\underbrace{\sqrt{a}z + \frac{b}{2\sqrt{a}}w + \frac{d}{2\sqrt{a}}}_{\lambda}\right)^2\\ + \underbrace{\left(c - \frac{b^2}{4a}\right)w^2 + \left(e - \frac{bd}{2a}\right)w + \left(f - \frac{d^2}{4a}\right)}_{\mu}.
\end{align*}

(ii) Let $a = 0$ and $c\neq 0$. Then 
\begin{align*}
cw^2 + bzw + dz + ew + f = \left(\underbrace{\sqrt{c}w  + \frac{b}{2\sqrt{c}}z + \frac{e}{2\sqrt{c}}}_{\lambda}\right)^2\\
+ \left(\underbrace{-\frac{b^2}{4c}z^2 + \left(d - \frac{be}{2c}\right)z + \left(f - \frac{e^2}{4c}\right)}_{\mu}\right)
\end{align*}

(iii) Let $a = 0,~ c = 0,~ b\neq 0$. It follows after application of a rotation $\begin{cases}
z = z' + w'\\
w = z' - w'
\end{cases}$ that \begin{align*}
bzw + dz + ew + f = \left(\underbrace{\sqrt{b}z' + \frac{d+e}{2\sqrt{b}}}_{\lambda}\right)^2 + \left(\underbrace{-bw'^2 + (d-e)w' + \left(f - \frac{(d+e)^2}{4b}\right)}_{\mu}\right).
\end{align*}

(iv) Let $a = b = c = 0$. Then $\lambda = 0$ and $\mu = dz + ew + f$.\\

Thus the determinant section is $\det(\phi) = q\otimes\mbox{Sym}^2(C)$ for some $q\in H^0(\mathcal{O}(2))$ and $C\in H^0(T_{\mathbb{P}^2}(-1))$. So the zero section lies in the image. Now let suppose that $q\otimes\mbox{Sym}^2(C) = q'\otimes\mbox{Sym}^2(C')$ for some nonzero global sections $q, q'$ and $C, C'$. This assumption is equivalent to $q\otimes C\otimes C = q'\otimes C'\otimes C'$. The zero schemes $\mathcal{Z}(q), \mathcal{Z}(q')$ are complex projective curves and the zero schemes $\mathcal{Z}(C), \mathcal{Z}(C')$ are two points on $\mathbb{P}^2$. Thus $\mathcal{Z}(q) = \mathcal{Z}(q')$, that is, $q' = \alpha^2\cdot q$ for some $\alpha\neq 0$. Thus $C\otimes C = (\alpha\cdot C')\otimes (\alpha\cdot C')$, i.e., $C = \pm\alpha\cdot C'$. Consider the action of $\mathbb{C}^\times$ on $H^0(\mathcal{O}(2))^\times\times H^0(T_{\mathbb{P}^2}(-1))^\times$ by $\alpha\cdot (q, C) = (\alpha^2q, \alpha^{-1}C)$. The quotient space characterizes the points in the image for which both $q$ and $C$ are nonzero. The zero section is obtained as a determinant if $\phi_0 = \begin{bmatrix}
    \lambda & \mu\\
    1 & -\lambda
\end{bmatrix}$ is nilpotent or $C = 0$. The quotient space along with the zero section is in bijection with the image of the determinant morphism by $q\otimes\mbox{Sym}^2(C)\mapsto [(q, C)]$. 
\end{proof}
\begin{remark}\label{dimcom}
Note that \begin{align*}
&\dim \mbox{Im}\left(\det|_{\mathcal{M}_{\mathbb{P}^2}(V_0^\rho, T_{\mathbb{P}^2})}\right) = \dim \left(\frac{H^0(\mathcal{O}(2))^\times\times H^0(T_{\mathbb{P}^2}(-1))^\times}{(q, C)\sim (\alpha^2\cdot q, \alpha^{-1}\cdot C)}\sqcup\{0\}\right)\\ & < \dim \left(H^0(\mathcal{O}(2))^\times\times H^0(T_{\mathbb{P}^2}(-1))^\times\right)
= 6 + 3 = 9 < 27 = \dim H^0(\mbox{Sym}^2(T_{\mathbb{P}^2})).
\end{align*} Thus $\det|_{\mathcal{M}_{\mathbb{P}^2}(V_0^\rho, T_{\mathbb{P}^2})}$ is not surjective onto $H^0(\mbox{Sym}^2(T_{\mathbb{P}^2}))$.
\end{remark}
\begin{theorem}\label{det2}
There is a bijection 
\begin{align*}
&\emph{Im}\left(\det|_{\mathcal{M}_{\mathbb{P}^2}(V_1^\rho, T_{\mathbb{P}^2})}\right)\to\\
&\left(\frac{H^0(\mathcal{O}(2))^\times\times H^0(T_{\mathbb{P}^2}(-1))^\times}{(q, C)\sim (\alpha^2\cdot q, \alpha^{-1}\cdot C)}\sqcup\frac{H^0(T_{\mathbb{P}^2})}{\{\pm 1\}}\right)\bigg/[(q, C)]\sim [A]\iff q\otimes\emph{Sym}^2(C) = \emph{Sym}^2(A)
\end{align*} given by $q\otimes\emph{Sym}^2(C)\mapsto [(q, C)]$ and $\emph{Sym}^2(A)\mapsto [A]$.
\end{theorem}

\begin{proof}
Since $V_1^\rho = \mathcal{O}\oplus\mathcal{O}$ is Gieseker semistable every traceless co-Higgs bundle $\phi$ modelled on $V_1^\rho$ is Gieseker semistable. Write 
\begin{equation}
\phi = \begin{bmatrix}
  A & B\\
    C & -A
\end{bmatrix},~~ A, B, C\in H^0(T_{\mathbb{P}^2}). 
\end{equation}
Let $B, C\neq 0$. The integrability condition of $\phi$ implies that either 
\begin{equation}
\phi = \begin{bmatrix}
    \lambda & \mu\\
    \mu' & -\lambda
\end{bmatrix}\otimes C',~~ \lambda, \mu, \mu'\in H^0(\mathcal{O}(1)), C\in H^0(T_{\mathbb{P}^1}(-1))
\end{equation}
or 
\begin{equation}
\phi = \begin{bmatrix}
    \lambda & \mu\\
    1 & -\lambda
\end{bmatrix}\otimes C ,~~\mu\neq 0, \lambda\in\mathbb{C}.
\end{equation}
In the second case, $\det(\phi) = -(\lambda^2 + \mu)\mbox{Sym}^2(C)$. \\

On the other hand let $B = 0$ or $C = 0$. Without loss of generality $C=0$. So $\phi = \begin{bmatrix}
    A & B\\
    0 & -A
\end{bmatrix}.$ The integrability condition of $\phi$ is that $A\wedge B = 0$. Irrespective of $B=0$ or $B\neq 0$ it follows that $\det(\phi) = -\mbox{Sym}^2(A)$.\\

Considering all possible cases of the determinants, it follows that \begin{align*}
&\mbox{Im}\left(\det|_{\mathcal{M}_{\mathbb{P}^2}(V_1^\rho, T_{\mathbb{P}^2})}\right)\\
&= \left\{q\otimes\mbox{Sym}^2(C): q\in H^0(\mathcal{O}(2)),~ C\in H^0(T_{\mathbb{P}^2}(-1))\right\}\bigcup \left\{\mbox{Sym}^2(A): A\in H^0(T_{\mathbb{P}^2})\right\}\\
&= \left\{q\otimes\mbox{Sym}^2(C): q\in H^0(\mathcal{O}(2))^\times,~ C\in H^0(T_{\mathbb{P}^2}(-1))^\times\right\}\bigcup \left\{\mbox{Sym}^2(A): A\in H^0(T_{\mathbb{P}^2})\right\}.
\end{align*}

However, the intersection of these two collections of determinant sections in the last step is nonempty. There are sections $q\in H^0(\mathcal{O}(2)),~ C\in H^0(T_{\mathbb{P}^2}(-1))$ and $A\in H^0(T_{\mathbb{P}^2})$ such that $$q\otimes\mbox{Sym}^2(C) = \mbox{Sym}^2(A).$$ Such sections should be identified as the same. On the other hand sections of the form $q\otimes\mbox{Sym}^2(C)$ for $q\neq 0,~ C\neq 0$ are identified with points in $\frac{H^0(\mathcal{O}(2))^\times\times H^0(T_{\mathbb{P}^2}(-1))^\times}{(q, C)\sim (\alpha^2\cdot q, \alpha^{-1}\cdot C)}$ by $q\otimes\mbox{Sym}^2(C)\mapsto [(q, C)]$ and the sections of the form $\mbox{Sym}^2(A)$ are identified with points in $H^0(T_{\mathbb{P}^2})$. Under the identification $q\otimes\mbox{Sym}^2(C) = \mbox{Sym}^2(A)$ we arrive at a bijective correspondence 

\begin{align*}
&\mbox{Im}\left(\det|_{\mathcal{M}_{\mathbb{P}^2}(V_1^\rho, T_{\mathbb{P}^2})}\right)\to\\
&\left(\frac{H^0(\mathcal{O}(2))^\times\times H^0(T_{\mathbb{P}^2}(-1))^\times}{(q, C)\sim (\alpha^2\cdot q, \alpha^{-1}\cdot C)}\sqcup\frac{H^0(T_{\mathbb{P}^2})}{\{\pm 1\}}\right)\bigg/[(q, C)]\sim [A]\iff q\otimes\mbox{Sym}^2(C) = \mbox{Sym}^2(A).
\end{align*}

via the correspondence $q\otimes\mbox{Sym}^2(C)\mapsto [(q, C)];~~\mbox{Sym}^2(A)\mapsto [A]$.
\begin{remark}
Now we have
\begin{align*}
&\dim \mbox{Im}\left(\det|_{\mathcal{M}_{\mathbb{P}^2}(V_1^\rho, T_{\mathbb{P}^2})}\right) = \max \left\{\dim \frac{H^0(\mathcal{O}(2))^\times\times H^0(T_{\mathbb{P}^2}(-1))^\times}{(q, C)\sim (\alpha^2\cdot q, \alpha^{-1}\cdot C)};~\dim \frac{H^0(T_{\mathbb{P}^2})}{\{\pm 1\}}\right\}\\
& < \max\{9, 8\} = 9 < 27 = \dim H^0(\mbox{Sym}^2(T_{\mathbb{P}^2})).
\end{align*} Thus $\det|_{\mathcal{M}_{\mathbb{P}^2}(V_1^\rho, T_{\mathbb{P}^2})}$ is not surjective onto $H^0(\mbox{Sym}^2(T_{\mathbb{P}^2}))$.
\end{remark}

\end{proof}

\begin{theorem}\label{det3}
There is a bijection $\emph{Im}\left(\det|_{\mathcal{M}_{\mathbb{P}^2}(V_2^\rho, T_{\mathbb{P}^2})}\right)\to\frac{H^0(\mathcal{O}(2))^\times\times H^0(T_{\mathbb{P}^2}(-1))^\times}{(q, C)\sim (\alpha^2q, \alpha^{-1}C)}\sqcup\{0\}$ given by $q\otimes\emph{Sym}^2(C)\mapsto [(q, C)]$.
\end{theorem}

\begin{proof}
Since $V_2^\rho = T_{\mathbb{P}^2}$ is stable any co-Higgs bundle $(T_{\mathbb{P}^2}, \phi)$ is Gieseker stable. It follows from the Euler sequence of $T_{\mathbb{P}^2}$ that there is an isomorphism of complex vector spaces $H^0(\mbox{End}_0T_{\mathbb{P}^2}\otimes T_{\mathbb{P}^2})\cong H^0(\mbox{End}_0 T_{\mathbb{P}^2}\otimes\mathcal{O}(1))\otimes H^0(T_{\mathbb{P}^2}\otimes\mathcal{O}(-1))$ (\cite{Rayan}, \cite{Ste3}). Thus an algebraic tensor product of nonzero sections $\phi_0\in H^0(\mbox{End}_0 T_{\mathbb{P}^2}\otimes\mathcal{O}(1))$ and $C\in H^0(T_{\mathbb{P}^2}\otimes\mathcal{O}(-1))$ is a section of $\mbox{End}_0T_{\mathbb{P}^2}\otimes T_{\mathbb{P}^2}$. This algebraic tensor product coincides with the tensor product section $\phi_0\otimes C$. To verify this fact let us choose a global section $\phi = \phi_0\otimes C$. Then $\det(\phi) = \det(\phi_0)\otimes\mbox{Sym}^2(C)$. The zero scheme of $\mathcal{Z}(\det(\phi_0))$ is either $\mathbb{P}^2$ or a conic and $\mathcal{Z}(C)$ is a point. In all cases, $C$ is unique up to a $\mathbb{C}^\times$ action, so is $\phi_0$ and their algebraic tensor product remains an invariant.\\

On the other hand, a co-Higgs field $\phi = \sum c_{ij}\phi_i\otimes C_j$ is integrable if and only if $\phi = \phi_0\otimes C$. Thus every integrable co-Higgs field modelled on $T_{\mathbb{P}^2}$ is $\phi = \phi_0\otimes C$. Moreover it follows that the restricted determinant morphism $\det: H^0(\mbox{End}_0 T_{\mathbb{P}^2}\otimes \mathcal{O}(1))\to H^0(\mathcal{O}(2))$ is surjective. It suffices to verify this fact on trivializing charts. As a complex manifold $\mathbb{P}^2$ is union of three open sets while on each of these open sets exactly one complex component is nonzero. The trivializing charts of $\mathbb{P}^2$ are given by
\begin{equation}
\begin{cases}
\psi_1([z:w:1]) = (z, w)\\
\psi_2([z:1:w]) = (z, w)\\
\psi_3([1:z:w]) = (z, w).
\end{cases}
\end{equation}
Letting $(z, w)$ be coordinates on the first trivializing chart $\psi_1$ the trivialization maps on $T_{\mathbb{P}^2}$ are given by

\begin{equation}
g_{12}' = \begin{bmatrix}
\frac{1}{w} & -\frac{z}{w^2}\\
0 & -\frac{1}{w^2}
\end{bmatrix};~ g_{23}' = \begin{bmatrix}
-\frac{1}{z^2} & 0\\
-\frac{w}{z^2} & \frac{1}{z}
\end{bmatrix};~ g_{31}' = \begin{bmatrix}
    -\frac{w}{z^2} & \frac{1}{z}\\
    -\frac{1}{z^2} & 0
\end{bmatrix}
\end{equation}

The local equation of a global section $\phi_0\in H^0(\mbox{End}_0(T_{\mathbb{P}^2})\otimes\mathcal{O}(1))$ is given by 
\begin{equation}
\begin{bmatrix}
F(z, w) & G(z, w)\\
H(z, w) & -F(z, w)
\end{bmatrix}\cdot\begin{bmatrix}
\frac{1}{w} & -\frac{z}{w^2}\\
0 & -\frac{1}{w^2}
\end{bmatrix} = \begin{bmatrix}
1 & -\frac{z}{w}\\
0 & -\frac{1}{w}
\end{bmatrix}\cdot\begin{bmatrix}
f(\frac{z}{w}, \frac{1}{w}) & g(\frac{z}{w}, \frac{1}{w})\\
h(\frac{z}{w}, \frac{1}{w}) & -f(\frac{z}{w}, \frac{1}{w})
\end{bmatrix}.
\end{equation}
In particular, the determinant of the matrix $\begin{bmatrix}
F(z, w) & G(z, w)\\
H(z, w) & -F(z, w)
\end{bmatrix}$ is a polynomial in $z, w$ of total degree $\leq 2$ which represents $\det(\phi_0)\in H^0(\mathcal{O}(2))$ on one coordinate neighbourhood.\\

A comparison of the coefficients in the power series of $F, G, H$ around $(0,0)$ with the respective power series of $f, g, h$ suggests that $H, h$ are constants, $F, f$ are polynomials in $z, w$ of degree $\leq 1$ and $G, g$ are polynomials in $z, w$ of total degree $\leq 2$. It is already observed that every complex polynomial of total degree $\leq 2$ is of the form $-F^2 - GH$. So the determinant morphism is surjective onto $H^0(\mathcal{O}(2))$. The same reasoning in the proof of Theorem \ref{det1} concludes that  $\mbox{Im}\left(\det|_{\mathcal{M}_{\mathbb{P}^2}(V_2^\rho, T_{\mathbb{P}^2})}\right)$ is in bijective correspondence with the space $$\frac{H^0(\mathcal{O}(2))^\times\times H^0(T_{\mathbb{P}^2}(-1))^\times}{(q, C)\sim (\alpha^2q, \alpha^{-1}C)}\sqcup\{0\}$$ via the correspondence $q\otimes\mbox{Sym}^2(C)\mapsto [(q, C)]$.
\end{proof}

\begin{remark}
An estimate of dimensions similar to the one given in Remark \ref{dimcom} suggests that $\det|_{\mathcal{M}_{\mathbb{P}^2}(V_2^\rho, T_{\mathbb{P}^2})}$ is not surjective onto $H^0(\mbox{Sym}^2(T_{\mathbb{P}^2}))$.
\end{remark}

\begin{theorem}\label{det4}
For $k>3$ there is a bijection $\emph{Im}\left(\det|_{\mathcal{M}_{\mathbb{P}^2}(V_k^\rho, T_{\mathbb{P}^2})}\right)\to\frac{H^0(T_{\mathbb{P}^2}(-1))}{\{\pm 1\}}$ given by $\emph{Sym}^2(A)\mapsto [A]$.
\end{theorem}
\begin{proof}
From an Euler sequence of $V_k^\rho$ it follows that there is an isomorphism of complex vector spaces $H^0(\mbox{End}_0(V_k^\rho)\otimes\mathcal{O}(1))\otimes H^0(T_{\mathbb{P}^2}(-1))\cong H^0(\mbox{End}_0(V_k^\rho)\otimes T_{\mathbb{P}^2})$. Here every $\phi$ is integrable and is of the form $\phi = \phi_0\otimes C$ where $\det(\phi_0) = \lambda\cdot \rho$ for some $\lambda\in\mathbb{C}^\times$. Hence $\det(\phi) = \lambda\rho\otimes\mbox{Sym}^2(C)$ and 

$$\mbox{Im}\left(\det|_{\mathcal{M}_{\mathbb{P}^2}(V_k^\rho, T_{\mathbb{P}^2})}\right) = \left\{\rho\otimes\mbox{Sym}^2(C): C\in H^0(T_{\mathbb{P}^2}(-1))\right\}.$$

Since $\rho\neq 0$ it follows that there is a bijection $\mbox{Im}\left(\det|_{\mathcal{M}_{\mathbb{P}^2}(V_k^\rho, T_{\mathbb{P}^2})}\right)\to\frac{H^0(T_{\mathbb{P}^2}(-1))}{\{\pm 1\}}$ given by $\mbox{Sym}^2(A)\mapsto [A]$.
\end{proof}

\begin{remark}

Finally, we have
\begin{align*}
&\dim \mbox{Im}\left(\det|_{\mathcal{M}_{\mathbb{P}^2}(V_k^\rho, T_{\mathbb{P}^2})}\right) = \dim \frac{H^0(T_{\mathbb{P}^2}(-1))}{\{\pm 1\}} \leq 3 < 27 = \dim H^0(\mbox{Sym}^2(T_{\mathbb{P}^2})).
\end{align*} Thus $\det|_{\mathcal{M}_{\mathbb{P}^2}(V_k^\rho, T_{\mathbb{P}^2})}$ is not surjective onto $H^0(\mbox{Sym}^2(T_{\mathbb{P}^2}))$ for $k>3$.
\end{remark}

\printbibliography

\end{document}